\documentclass[11pt,oneside,reqno]{amsart}
\usepackage{amsfonts}
\usepackage{amssymb,amsmath, amsthm}
\usepackage[numbers]{natbib}
\usepackage{graphicx}
\usepackage[pdftex,plainpages=false,colorlinks,hyperindex,bookmarksopen,linkcolor=red,citecolor=blue,urlcolor=blue]{hyperref}
\usepackage{mathrsfs}
\usepackage{xcolor}
\usepackage{epstopdf}
\usepackage{color}
\usepackage{enumerate}
\usepackage{float}
\usepackage{enumitem}
\usepackage{mathtools}
\usepackage{lipsum}
\usepackage{comment}
\calclayout
\makeatletter
\def\author@andify{  \nxandlist {\unskip ,\penalty-1 \space\ignorespaces}    {\unskip {} \@@and~}    {\unskip \penalty-2 \space \@@and~}}
\makeatother
\pdfoutput=1
\bibpunct{[}{]}{;}{n}{,}{,}
\newtheorem{theorem}{Theorem}[section]

\theoremstyle{remark}
\newtheorem{remark}{Remark}[section]

\newtheorem{lem}{Lemma}[section]

\newcommand{\be}{\begin{eqnarray}}
\newcommand{\ee}{\end{eqnarray}}

\numberwithin{equation}{section}
\allowdisplaybreaks

\begin{document}

\title{Shot-noise processes with logarithmic response function and their scaling limits}
\date{}

\address{${}^1$ Department of Statistical Sciences, Sapienza University of
Rome}
\email{luisa.beghin@uniroma1.it}
\email{enrico.scalas@uniroma1.it}
\address{${}^2$ Department of Mathematics, Luxembourg University}
\email{lorenzo.cristofaro@uni.lu}

\author{}

\author{}

\author{Luisa Beghin$^{1}$} \author{Lorenzo Cristofaro$^{2}$} \author{Enrico Scalas$^{1}$} 

\begin{abstract}
We consider shot-noise processes with an impulse response written in terms of the logarithm of the ratio between current and event time (instead of the usual absolute time difference). We study its finite-time properties as well as its weak convergence, under appropriate scaling and with general assumptions on the dependence of noises on event times. The limiting process coincides with the so-called Hadamard fractional Brownian motion (introduced in \cite{BEG}), which represents a middle ground between standard Brownian motion and fractional Brownian motion. It shares with the former the one-dimensional distribution (i.e. Gaussian with the same first two moments), while possessing the long-memory property (within a certain parameter range) of the latter, though with smaller intensity. Therefore, we identify a natural probabilistic scheme based on shot-noise processes whose scaling limit is the Hadamard fractional Brownian motion, thereby providing a concrete stochastic finite-time counterpart of this process.


 \bigskip
\end{abstract}
\maketitle
\section{Introduction}
Shot-noise processes were introduced by Schottky to explain noise in vacuum tubes due to the random arrival of discrete charge carriers \cite{schottky1918}. The concept was then used in other fields also outside physics. Nowadays, they represent a versatile tool that can be used to model a variety of phenomena in physics, biology, finance, and so on; they are particularly useful for modeling situations in which there are sudden, discrete changes in a system. These processes are often applied, in finance, to model the fluctuations in asset prices due to the arrival of trades. In this context, the \textit{shots} are the individual trades, and the \textit{noise} is the overall fluctuation in the asset price. For a general reference on this kind of application, see \cite{SCH}, where, in particular, time-inhomogeneous shot-noise processes are treated. Further applications, for example, to actuarial sciences are explored in \cite{DAS} on pricing catastrophe reinsurance and derivatives using Cox processes with shot noise intensity. In \cite{KLU2} explosive Poisson shot noise processes are studied with applications to risk reserves' management.  

To give a better idea on shot-noise, it is useful to resort to heuristic considerations. Let $\left\{J_i\right\}_{i>0}$ be a sequence of i.i.d. exponential random variables of parameter $\lambda>0$ with the meaning of waiting times between events and define the random walk $T_n$ with $T_0 =0$
\begin{equation}
    T_n = \sum_{i=1}^n J_i
\end{equation}
and its inverse, the renewal Poisson counting process,
\begin{equation}
    N(t) = \max \{n: T_n \leq t \}.
\end{equation}
Let $\left\{Y_i\right\}_{i>0}$ be a sequence of i.i.d. random variables and define the shot noise as
\begin{equation}
    Y(t) = \sum_{1\leq i \leq N(t)} Y_i \delta(t-T_i), \qquad t \geq 0.
    \label{shotnoise}
\end{equation}

A general non-stationary linear system (operator), $T$, can be represented by a Volterra kernel \cite{volterra1, volterra2, volterra3}. Given the informal nature of this introduction we do not discuss the domain of $T$ in detail, but we assume that it can be applied to (tempered) distributions. Let $y(t)$ represent the input signal and $s(t)$ the output signal, then we have
\begin{equation}
    s(t) = T[y(t)] = s_0 + \int_0^t g_T (t,\tau) y(\tau) \, d \tau,
    \label{linearsystem}
\end{equation}
where $g_T (t,\tau)$ is the impulse response of the system. Now, if we set $y(t) = Y(t)$, we get the following process $S(t)$ as output of the linear system 
\begin{equation}
    S(t) = \sum_{1\leq i \leq N(t)} Y_i g_T (t,T_i).
    \label{responsetoshot}
\end{equation}
Calling $S(t)$ {\em shot noise} is perhaps an abuse of language as the name should be reserved to the process $Y(t)$ of equation \eqref{shotnoise} before entering the linear system, but it is now common in the mathematics literature. An advantage of dealing with objects of the form \eqref{responsetoshot} is that one can avoid the theory of (tempered) distributions and focus on ``standard'' processes instead. Moreover, strictly speaking, an equation such as \eqref{shotnoise} is purely ``theoretical'' as what one observes is filtered through a measurement apparatus whose action can be often approximated by a linear system. As an example, if in equation \eqref{linearsystem}, we set $g_T (t,\tau) = 1$, the output is just the integral of the input $y(t)$:
$$
s(t) = s_0 + \int_0^t y(\tau) \, d\tau.
$$
In the special case of shot noise input \eqref{shotnoise}, we get
\begin{equation}
    S(t) = \sum_{i=1}^{N(t)} Y_i,
    \label{CTRW}
\end{equation}
namely a compound Poisson process. It is not necessary to assume that the linear system defined in equation \eqref{linearsystem} is conservative or dissipative. In other words, it is possible to have an increasing impulse response as a consequence of external energy introduced into the system, such as in the case of an amplifier.

Recently, some papers aimed to establish either
self-similar
stationary Gaussian or stable
limits of shot-noise processes under different assumptions (see \cite{KLU3}).

It is well-known that Poisson shot noise, under a regular variation
condition on the shot shape function, converges weakly to fractional Brownian motion with Hurst parameter greater than $1/2$ (see \cite{KLU}) and thus asymptotically displays a long-range dependence. 

This result has been recently extended in \cite{PAN} considering integrated shot-noise processes where noises are conditionally independent given their arrival times, but have distribution depending on the latter. In this case, the limiting processes obtained are self-similar, Gaussian, with non-stationary increments and generalize the fractional Brownian motion (hereafter fBm). The sample path moderate deviations for the same model have been studied in \cite{WAN}.  The weak convergence of (scaled) renewal shot noise processes, with response functions which are eventually non-decreasing and regularly varying at infinity, has been considered in \cite{IKS}. The renewal shot noise is characterized by interarrival times that are not necessarily exponentially distributed, but may follow any distribution on the positive half-line. The asymptotic behavior of renewal shot noise process with random response functions and immigration has been studied in \cite{IKS2} and \cite{IKS3}. 

Here, we consider Poisson shot noise under a proper definition of the shot space function (the impulse response), i.e. assuming that it depends on the ratio, instead of on the difference of times, and with a slowly-varying behavior. We analyze its finite-time properties as well as the asymptotic behavior. The most relevant feature is the following: At least in the simplest case of independent noises,
the quadratic variation of the process is finite and equal to zero, contrary to what happens (under similar assumptions) for shot noise with regularly-varying shot space function. 

Moreover, we prove that the scaling limit of this process  coincides with a generalization of fBm, called Hadamard-fBm and denoted as $B_\alpha^H:=\left\{B_\alpha^H(t) \right\}_{t \geq 0}$. The latter process was introduced and studied in \cite{BEG}: it is a centered, Gaussian process with the same one-dimensional distribution (i.e. with the same first two moments) as Brownian motion, but with  non-stationary increments displaying either short-range or long-range dependence, for $\alpha \in (0,1)$ and $\alpha \in (1,2)$, respectively. Many other properties of the process (i.e. H\"{o}lder continuity, quasi-helix behavior, power variation, local nondeterminism) are studied in \cite{BDM} where its Wiener integration theory is developed. 

This kind of diffusion processes (in the case $\alpha \in (1,2)$) are potential candidates to model financial time-series with standard mean square displacement, but with long-range dependence and non-stationary increments. We show that the shot noise defined here, in the special case of independence between noises and epochs, possesses scale-invariant autocorrelation function (with time lag $\tau$), even in the finite time-horizon; this means that it is well-suited for modeling ageing systems with such a feature. 

In addition to establishing weak convergence, the main contribution of this paper is to
provide a natural probabilistic scheme whose scaling limit coincides with the Hadamard-fBm. More precisely, we show that a class of shot-noise processes with logarithmic
response function, under suitable normalization, converges to it. In this sense, our model
can be interpreted as a constructive approximation of the Hadamard-fBm,
analogous to the classical approximation of fractional Brownian motion via Poisson shot-noise
processes with regularly varying kernels.

This perspective complements the existing literature, where H-fBm is introduced via white-noise constructions, by providing a concrete stochastic particle-based model
whose large-scale behavior yields the same limit.

\section{Definitions and preliminary results}
In view of what follows, we recall \textit{Tricomi's confluent hypergeometric function} (see \cite{NIST}, formula (13.2.42)) defined as
\begin{equation}\Psi\left(a, b ; z\right):=\frac{\Gamma(1-b)}{\Gamma(1+a-b)}\Phi(a,b;z)+\frac{\Gamma(b-1)}{\Gamma(a)}\Phi(1+a-b,2-b;z), \label{tri}
\end{equation}
for $a,b,z \in \mathbb{C}$, $\Re(b) \neq 0, \pm 1, \pm 2,...$, where $\Phi(a,b;z):=\sum_{l=0}^{+\infty}\frac{(a)_l}{(b)_l}\frac{z^l}{l!}$ and $(c)_l :=\frac{\Gamma(c+l)}{\Gamma(c)}$. 

The following integral representation holds for the confluent hypergeometric function: 
\begin{equation} \label{eqTricomiINtegral} \Psi\left(a, b ; z\right)=\frac{1}{\Gamma(a)}\int^{\infty}_{0}e^{-sz}s^{a-1}(1+s)^{b-a-1}ds, \end{equation} if $\Re(a)>0,$ $\Re(z) \geq 0$ (see \cite{KIL}, p.\ 30), 
while the asymptotic behavior, for $z \rightarrow 0$, is given by 
\begin{equation}
\Psi(a,b;z)=\frac{\Gamma(1-b)}{\Gamma(a-b+1)}+O(z), \qquad \Re(b) <0 \label{asy0}
\end{equation}
(see \cite{NIST}, formula (13.2.22)).

Moreover, one can check that the function $\Psi(a,b;\cdot)$ is non-decreasing, for $a<0$, since 
\begin{equation}
\frac{d}{dx}\Psi(a,b;x)=-a\Psi(a+1,b+1;x) \label{derpsi}
\end{equation}
(see \cite{NIST}, formula (13.3.22) for $n=1$).

Let us recall the definition, given in \cite{BEG}, of the so-called ``Hadamard fractional Brownian motion" (hereafter H-fBm), in the white-noise space $(\mathcal{S}'(\mathbb{R}), \mathcal{B}, \nu)$ (with $\mathcal{S}'(\mathbb{R})$ being the dual of Schwartz space $\mathcal{S}(\mathbb{R})$, $\mathcal{B}$ its cylinder $\sigma$-algebra and $\nu(\cdot)$ the white noise measure), as $B^H_\alpha:=\left\{ B^H_\alpha (t)\right\}_{t \geq 0}$, where $B^H_\alpha (t,\omega):=\langle \omega, \thinspace ^H\mathcal{M}^{\alpha/2} _{-}1_{[0,t)} \rangle$, $t\geq 0,\;\omega \in \mathcal{S%
}^{\prime }(\mathbb{R})$, and
\begin{equation*}
\left( \thinspace ^H\mathcal{M}_{-}^{\alpha/2 }f\right) (x):=\left\{
\begin{array}{l}
K_\alpha \left( \thinspace ^H\mathcal{D}_{-}^{(1-\alpha)/2 }f\right) (x),\qquad \alpha \in (0,1)
\\
f(x),\qquad \alpha =1 \\
K_\alpha \left( \thinspace ^H\mathcal{I}_{-}^{(\alpha-1)/2 }f\right) (x),\qquad \alpha \in (1,2).%
\end{array}%
\right.
\end{equation*}
where $K_{\alpha}=\Gamma ((\alpha+1)/2)/\sqrt{\Gamma(\alpha)}$ and $^H\mathcal{D}_{-}^{(1-\alpha)/2 }$, $^H\mathcal{I}_{-}^{(\alpha-1)/2 }$ are the Hadamard fractional derivative and fractional integral, respectively (see \cite{BEG} for details and further references).

Let now $C_\alpha =2^{1-\alpha}\sqrt{\pi}/\Gamma(\alpha/2)$, then, it was proved in \cite{BEG} that, for any $\alpha \in (0,1)  \cup (1,2)$, the H-fBm is a centered Gaussian process, with characteristic function given, for $0<t_1<...<t_n$, $n \in \mathbb{N}$, and $k_j \in \mathbb{R},$ $j=1,...,n$, by 
\begin{equation}
    \mathbb{E}e^{i\sum_{j=1}^{n}k_j B^H_{\alpha }(t_j)}=\exp \left\{-\frac{1}{2}\sum_{j,l=1}^{n}k_j k_l \sigma^\alpha_{j,l}   \right\}, \label{cf4}
\end{equation}
where 
\begin{equation} 
\sigma^\alpha_{j,l}:=\left\{
\begin{array}{l}
t_j , \qquad j=l \\
C_\alpha (t_j \wedge t_l)\Psi\left(\frac{1-\alpha}{2}, 1-\alpha ; \log\left(\frac{t_j \vee t_l}{ t_j \wedge t_l}\right) \right) , \qquad j \neq l. \label{sigma}
\end{array}
\right.
\end{equation}
Moreover, $B^H_\alpha$ has non-stationary increments, with variance 
\begin{eqnarray}\label{rho} 
\rho (s,t):=\mathbb{E}\left[ B_\alpha ^H (t)-B_\alpha ^H (s)\right]^2 
&=&\frac{1}{\Gamma(\alpha)}\int_{0}^{t} \left[\left( \log \frac{t}{x}\right)_+
^{(\alpha-1)/2 }-\left( \log \frac{s}{x}\right)_+
^{(\alpha-1)/2 }\right]^2 dx \\
&=&t+s -2C_\alpha   s\Psi\left(\frac{1-\alpha}{2}, 1-\alpha ; \log\left(\frac{t }{ s}\right) \right) \notag
\end{eqnarray}
for $0\leq s<t$.

It is proved in \cite{BEG} that $B_\alpha^H$ is self-similar with index $1/2$, it is stochastically continuous, for any $\alpha \in (0,1) \cup (1,2)$, and has $a.s.$ continuous sample paths for $\alpha \in (1,2)$; H\"{o}lder properties of trajectories can be found in \cite{BDM}.

Finally, 
the process is proved to be long-range dependent, for $\alpha \in (1,2)$ (in \cite{BEG}), while it has short memory, for $\alpha \in (0,1)$ (see \cite{BDM}).

The case $\alpha \in (1,2)$ is therefore particularly interesting both for the a.s. continuity and the long-range dependence properties. The key tool for the proof of the latter is the variance of the increments, which is also crucial for our analysis. 
We then present the following lemma, which summarizes the properties of $\rho(\cdot,\cdot)$:

\begin{lem}\label{lem1}
For $\alpha \in (1,2)$, the function $\rho(\cdot,\cdot)$ given in \eqref{rho} has the following properties:
\begin{enumerate}[label=(\emph{\roman*})]
    \item $\rho(0,0)=0$ and $\rho(s,t) \geq 0$, for any $0\leq s \leq t$
    \item it is ``super-additive": $\rho(r,s) + \rho(s,t) \leq \rho(r,t)$, for any $0  \leq r \leq s \leq t$
    \item  $\rho(s,t) \leq \rho(s,t')$ and  $\rho(s,t') \geq \rho(t,t')$, for any $0 \leq s  \leq t \leq t'$
    \item $\lim_{h \to 0} \rho(t,t+h)=0$ 
    \item $\rho(0,t)$ is continuous at $t.$
\end{enumerate}
\end{lem}

\begin{proof} Let us consider the above points, one by one.
\begin{enumerate}[label=(\emph{\roman*})]
\item follows from the first line of \eqref{rho}.
 \item was already proved in \cite{BEG}; see Theorem 3.2.
 \item from (\emph{i}) and (\emph{ii}) we immediately have that $\rho(s,t) \leq \rho(s,t')$; on the other hand, from the first line of \eqref{rho},  we get $\rho(s,t') \geq \rho(t,t')$.
 \item follows from the second line of \eqref{rho}, using \eqref{asy0}, together with the duplication formula of the gamma function.
 \item is an immediate consequence of \eqref{rho}.
 \end{enumerate}  
 
\end{proof}

\section{Main results}
We consider here the shot-noise process defined as $S:=\left\{S(t)\right\}_{t \geq 0}$, where
\begin{equation}
S(t):=\sum_{j=1}^{+\infty} g(t/T_j) R_j(T_j), \qquad t \geq 0, \label{defs}
\end{equation}
and $g: \mathbb{R}^+ \to \mathbb{R}^+$ represents the response function, which, in this case, depends on the ratio instead of the distance between the current time and the renewal epoch. We assume that $g(t)=0,$ for $t \leq 1$, in order to guarantee convergence of the sum in \eqref{defs}; moreover, 
since $g(\cdot)$ has non-negative values, the effect of the noises is supposed to linger over time. The r.v.'s $T_j$, $j \in \mathbb{N}$, are the Poisson arrival times (epochs) of shots, with rate $\lambda$; for any $j \in \mathbb{N}$, the r.v. $R_j(T_j)$ represents the noise (a random variable) caused by the $j$-th shot at time $T_j$ and is assumed to depend on $T_j$ (as emphasized by the notation). On the other hand, the $R_j$'s are conditionally independent of each other and identically distributed with 
\[
F_u(x):=P(R_j \leq x |T_j=u), \qquad x \in \mathbb{R}, u \in \mathbb{R}^+.
\]
The conditional moments of $R_j$, i.e. $K_r(u):=\int_{\mathbb{R}}x^r dF_u(x)$, for any $u,r \in \mathbb{R}^+ $ and $j \in \mathbb{N}$, are assumed to be finite at least for $r=1,2,3,4$. Moreover, we assume that $K_1(u)=0,$ for any $u,$ without loss of generality.

Following \cite{BUT} and using the relation $\sum_{j=1}^{N(t)}g(t/T_j)=\int_0^{+\infty}g(t/u)N(du)$ (where $\left\{N(t)\right\}_{t \geq 0}$ is a Poisson process with intensity $\lambda$), we can write, for any $t \geq 0$,
\begin{equation}
S(t) = \int_0^{+\infty} \int_{\mathbb{R}} g(t/u) r N(du,dr), \label{sn1}
\end{equation}
where $N(\cdot,\cdot)$ is a Poisson measure with intensity $\lambda F_u(dr)du$. In view of what follows, we denote as $\overline{N}_a(du)$ the Poisson random measure with intensity $\lambda adu$.

Let the functions $h,h_1,h_2: \mathbb{R} \to \mathbb{R}$ be such that the integrals
\begin{eqnarray}
    &&\mathbb{E} \left[ \int \int h(u)r^p N(du,dr)\right]=\lambda \int h(u)K_p(u)du \notag \\
    &&\mathbb{E} \left[\int \int \int h_1(u) h_2(u)rr' N(du,dr)N(du,dr')\right]=\lambda \int h_1(u) h_2(u)K_2(u)du \notag
\end{eqnarray}
converge. Using the second formula with $h_1(u):=g(t/u)$ and $h_2(u):=g(s/u)$,
we can evaluate the auto-covariance function of $S$ as
\begin{equation}
\mathrm{Cov}(S(t),S(s))=
\lambda \int_0^{+\infty}g(t/u)g(s/u)K_2(u)du \label{cov} 
\end{equation}
for $s,t \geq 0$, 
(see \cite{PAN}, for details).
Similarly, we can obtain, from the first formula (with $p=1$ and $h(u):=\sum_{j=1}^n k_jg(t_j/u)$), the following characteristic function  
\begin{eqnarray}
    &&\mathbb{E}e^{i\sum_{j=1}^{n}k_j S(t_j)}=\mathbb{E}\exp \left\{i \int_0^{+\infty} \int_{\mathbb{R}} \sum_{j=1}^{n}k_j g(t_j/u) r N(du,dr) \right\} \label{car} \\
    &=& \exp \left\{ \lambda\int_0^{+\infty} \int_{\mathbb{R}} \left[e^{i\sum_{j=1}^{n}k_j g(t_j/u) r}-1
    \right]F_u(dr)du \notag
    \right\},
\end{eqnarray}
 for 
$0<t_1<...<t_n<t$, $n \in \mathbb{N}$, and $k_j \in \mathbb{R},$ $j=1,...,n$.

We assume that 
\begin{equation}
g(t)=\log^{\beta}(t)1_{[1,+\infty)}(t), \label{hypg2}
\end{equation}
for $\beta \in (0,1/2)$ and $t \geq 0$, so that definition \eqref{defs} reduces to
\begin{eqnarray}\label{defs2}
S_\beta(t)&=&\sum_{j=1}^{+\infty} \log^{\beta}(t/T_j) R_j(T_j)1_{t \geq T_j}= \sum_{j=1}^{+\infty} (\log t - \log T_j)_+^\beta R_j(T_j)\\
&=&\sum_{j=1}^{N(t)} (\log t - \log T_j)^\beta R_j(T_j).\notag
\end{eqnarray}

Let $L(\cdot)$ be a positive slowly varying function at $+ \infty$ if $\lim_{x \to +\infty}L(ax)/L(x)=1$, for any $a \in \mathbb{R}^+$; we recall that  $\log ^\beta(x)$ is slowly varying, for any $\beta \in \mathbb{R}$, so that we are assuming here that the lingering effect of the noises, which is represented by $g(\cdot) $, is slowly varying over time. 
Moreover, we shall assume that the conditional variance of the noises, given the Poisson epochs, is asymptotically constant, when the latter tend to infinity.

\subsection{Finite-time results in some relevant cases}
We give some examples of the shot-noise process $S_\beta:=\left\{S_\beta(t)\right\}_{t\geq 0}$ under different assumptions on the conditional distribution (or variance) of the variables $R_j$'s. As we shall see, in these special cases, we can compute an explicit expression of the auto-covariance function of $S_\beta$.
\subsubsection*{Example 3.1 (Independent noises)}\label{ex1}
If we assume that the $R_j$'s are independent of the occurring times $T_j$'s, i.e. their distribution function is $F_u(\cdot)=F(\cdot)$, for any $u \geq 0$, and their variance is a constant: $K_2(\cdot)=K_2>0$, then the auto-covariance of $S_\beta$ follows from \eqref{cov} and \eqref{eqTricomiINtegral}, for $0\leq s<t$:
\begin{eqnarray}
\mathrm{Cov}(S_\beta(t),S_\beta(s))&=&K_2
\lambda \int_0^{+\infty}\left( \log \frac{s}{x}\right)_+
^\beta \left( \log \frac{t}{x}\right)_+
^\beta dx 
=K_2 \lambda\int_{0}^{s} \left( \log \frac{s}{x}\right)
^\beta \left( \log \frac{t}{x}\right)
^\beta dx \notag \\
&=&sK_2 \lambda\int_{0}^{\infty} \left[ \log \frac{t}{s}+w\right]
^\beta w^\beta e^{-w} dw  \label{ind} \\
&=&sK_2 \lambda \left[ \log \frac{t}{s}\right]^{2\beta +1 }\int_{0}^{\infty} (1+y)^{\beta } y^{\beta} e^{-y \log(t/s)} dy \notag \\
&=&\Gamma (\beta+1 )sK_2\lambda\left[ \log \frac{t}{s}\right]^{2\beta +1}\Psi\left(\beta+1, 2\beta +2 ; \log\frac{t}{ s} \right)  \notag \\
&=&\Gamma (\beta+1 )sK_2\lambda\Psi\left(-\beta, -2\beta ; \log\frac{t}{ s} \right), \notag
\end{eqnarray}
where, in the last step, we have used the following formula
\begin{equation}
\Psi(a,b;z)=z^{1-b}\Psi(a+1-b,2-b;z), \notag
\end{equation}
(see \cite{NIST}, eq.\ (13.2.40)). By recalling \eqref{asy0}, we get that $\mathrm{Var}(S_\beta(t))=\lambda K_2 \Gamma(2\beta+1)t$. As a consequence of \eqref{derpsi}, we can say that the auto-covariance is non-decreasing in $t$ and depends on the ratio of $s$ and $t$, instead of their difference. 

Thus the process is non-stationary, but it shares the covariance homogeneity (of index $1$) under multiplicative time scaling with Brownian motion, i.e. $\mathrm{Cov}(S_\beta(ct),S_\beta(cs))=c\thinspace \mathrm{Cov}(S_\beta(t),S_\beta(s))$, for any $s,t,c >0.$

By considering the asymptotic behavior of the Tricomi confluent hypergeometric function \eqref{tri}, for large arguments, i.e. $\Psi(a,b;z) \sim z^{-a}$, for $z\to \infty$ and $a \in \mathbb{R}$ (see \cite{NIST}, formula (13.7.10)), we can deduce that
\[
\mathrm{Cov}(S_\beta(t+\tau),S_\beta(t))\sim (\log (1+\tau/t))^\beta, \qquad \tau \to \infty.
\]
Thus the covariance decays extremely slowly, i.e. only logarithmically, which suggests that an ultra-long memory, much slower than any power law, should appear in the limit.

We can also see that the auto-correlation function displays a scale-invariance property: indeed we can write that
\[
\mathrm{Corr}(S_\beta(t),S_\beta(t+\tau))=C(1+\tau/t)^{-1/2}\Psi(-\beta,-2\beta; \log(1+\tau/t)),
\]
for a constant $C>0,$ and thus there exists a function $\phi(\cdot)$ such that
$\mathrm{Corr}(S_\beta(t),S_\beta(t+\tau))=C\phi(\tau/t)$. Usually these autocorrelation functions are used to describe long-time behavior in an aging system, where both the age $t$ and the time lag $\tau$ are large compared to the system’s intrinsic time scales. Thus it appears, asymptotically, in a large number of anomalous diffusion problems (see, among the others, \cite{DEC} where the case of non-stationary scale-invariant systems is considered). The process $S_\beta$, on the other hand, possesses scale-invariant autocorrelation for any finite $t$ and $\tau$ (not only asymptotically), when the noises are independent of the occurring times.

\subsubsection*{Example 3.2 (Noises with power-law decaying variance)}
If we assume that the conditional variance has a power law decay, i.e. $K_2(t)=K+1/t^\gamma$, for $K>0$ and $\gamma \in [0,1)$, we obtain, by calculations similar to the previous case, that, for $0\leq s<t$,
\begin{eqnarray}
\mathrm{Cov}(S_\beta(t),S_\beta(s))&=&
\lambda \int_0^{+\infty}\left( \log \frac{s}{x}\right)_+
^\beta \left( \log \frac{t}{x}\right)_+
^\beta \left(K+\frac{1}{x^\gamma} \right) dx 
 \notag \\
 &=&\Gamma (\beta+1 )s\lambda \left[K \Psi\left(-\beta, -2\beta ; \log\frac{t}{ s} \right)+\frac{(1-\gamma)^{-2\beta-1}}{s^\gamma}\Psi\left(-\beta, -2\beta ;(1-\gamma) \log\frac{t}{ s}\right)\right]. \notag 
\end{eqnarray}

\subsubsection*{Example 3.3 (Noises with logarithmically decaying variance)} If we assume that the conditional variance has the following logarithmic decay $K_2(t)=K+\log(t^{-\gamma})$, for $t>1$, $K>0$ and $\gamma>0$, we have instead that
\begin{eqnarray}
&&\mathrm{Cov}(S_\beta(t),S_\beta(s))=
\lambda \int_0^{+\infty}\left( \log \frac{s}{x}\right)_+
^\beta \left( \log \frac{t}{x}\right)_+
^\beta \left(K+\log(x^{-\gamma})\right) dx 
 \notag \\
 &=&  \lambda s K\int_{0}^{\infty} \left[ \log \frac{t}{s}+w\right]
^\beta w^\beta e^{-w} dw +\lambda s \gamma\int_{0}^{\infty} \left[ \log \frac{t}{s}+w\right]
^\beta w^\beta (w-\log s) e^{-w} dw\notag \\
&=&\Gamma (\beta+1 )\lambda s\left[K-\gamma\log s\right] \Psi\left(-\beta, -2\beta ; \log\frac{t}{ s} \right)+\Gamma (\beta+2 )\lambda s \gamma \Psi\left(-\beta, -2\beta-1 ;\log\frac{t}{ s}\right) \notag \\
&=& \Gamma (\beta+1 )\lambda sK_2(s) \Psi\left(-\beta, -2\beta ; \log\frac{t}{ s} \right)+\Gamma (\beta+2 )\lambda s \gamma \Psi\left(-\beta, -2\beta-1 ;\log\frac{t}{ s}\right) \notag
\end{eqnarray}

\begin{remark}
    Under the assumption of independent noises, as in Example 3.1, we can compute the quadratic variation of the process $S_{\beta}$ and compare it to the case of shot-noise with power response function, i.e. $X_\beta:=\left\{X_\beta(t) \right\}_{t \geq 0}$, where $X_\beta(t):=\sum_{j=1}^{+\infty} (t -  T_j)_+^\beta R_j$, under the same assumptions on $\beta$ and on the variance. Using equation \eqref{ind} together with the covariance homogeneity (of index $1$) and the results in \cite{BDM}, Theorem 5.1 with $p=2$, we have that, for $T>0$ and $\beta \in (0,1/2)$,
   \begin{eqnarray}
   &&\lim_{n \to \infty}\sum_{k=1}^n \mathbb{E} \left[ S_\beta\left( \frac{Tk}{n}\right)-S_\beta\left( \frac{T(k-1)}{n}\right)\right]^2 \notag \\
   &=& T\lim_{n \to \infty}\frac{1}{n}\sum_{k=1}^n \mathbb{E} \left[ S_\beta\left( k\right)-S_\beta\left(k-1\right)\right]^2 \notag \\
   &=&T\lambda K_2 \lim_{n \to \infty}\frac{1}{n}\sum_{k=1}^n \int_0^{k-1}\left[  \log^\beta (k/u)-\log^\beta ((k-1)/u) \right]^2 du+T\lambda K_2  \lim_{n \to \infty}\frac{1}{n}\sum_{k=1}^n \int_{k-1}^{k}  \log^{2\beta} (k/u) du \notag \\
   &\leq&2T\lambda K_2 \lim_{n \to \infty}\frac{1}{n}\sum_{k=1}^n  (k-1)^{-2\beta} =0, \notag
   \end{eqnarray}
since  $\sum_{j=1}^{n-1}j^a \sim n^{a+1},$ as $n \to \infty.$

On the other hand. for the standard case, we have that, for $s<t$ and for any $c>0$,
\begin{eqnarray}
\mathrm{Cov}(X_\beta(ct),X_\beta(cs))&=&
\lambda K_2 \int_0^{+\infty}( cs-x)_+
^\beta ( ct-x)_+
^\beta  dx= c^{2\beta+1}\mathrm{Cov}(X_\beta(t),X_\beta(s)).
 \notag
 \end{eqnarray}
 Therefore the quadratic variation reads
\begin{eqnarray}
   &&\lim_{n \to \infty}\sum_{k=1}^n \mathbb{E} \left[ X_\beta\left( \frac{Tk}{n}\right)-X_\beta\left( \frac{T(k-1)}{n}\right)\right]^2 \notag = T^{2\beta+1}\lim_{n \to \infty}\frac{1}{n^{2\beta+1}}\sum_{k=1}^n \mathbb{E} \left[ X_\beta\left( k\right)-X_\beta\left(k-1\right)\right]^2 \notag \\
   &=&T^{2\beta+1} \lim_{n \to \infty}\frac{1}{n^{2\beta+1}}\sum_{k=1}^n \int_0^{k-1}\left[ (k-u)^\beta-(k-1-u)^\beta \right]^2 du+T^{2\beta+1} \lim_{n \to \infty}\frac{1}{n^{2\beta+1}}\sum_{k=1}^n \int_{k-1}^{k}  (k-u)^{2\beta} du \notag \\
   &=&T^{2\beta+1} \lim_{n \to \infty}\frac{1}{n^{2\beta+1}}\sum_{k=1}^n \int_0^{k-1}\left[ (1+v)^\beta-v^\beta \right]^2 dv+\frac{T^{2\beta+1}}{2\beta+1} \lim_{n \to \infty}n^{-2\beta} \notag \\
   &\sim&2T^{\beta+1}\lim_{n \to \infty}\frac{1}{n^{2\beta+1}}\sum_{k=1}^n  (k-1), \notag
   \end{eqnarray}
   which diverges for $n \to \infty,$ giving infinite quadratic variation in the standard case, contrary to the zero quadratic variation of $S_\beta.$

   Below, we present two plots where the trajectories of the process $S_\beta$ are compared to those of $X_\beta$, for different values of $\beta$ and under the same assumptions on Poisson intensity and jump distribution (i.e. for $\lambda=K_2=1$ and for standard Gaussian $R_j$'s). They confirm the results obtained on the increment variance and on quadratic variation: already for $\beta=0.1$, the polynomial kernel shows a greater cumulative variability compared to the logarithmic one. For $\beta=0.3$ this is much more evident: the trajectories of $S_\beta$ are considerably regular, while the weight of past jumps causes more sudden changes in the trajectories of $X_\beta$. Note that, in all these simulations, the $T_j$s and the $R_j$s are not changed for ease of comparison. The code used for the simulations can be found at the following GitHub link: {\tt https://github.com/FractionalEarthquakes/HfBm}.
\begin{figure}[H]
\caption{Logarithmic vs polynomial shot-noises, for $\beta=0.1$}
{{\includegraphics[width=10.5cm]{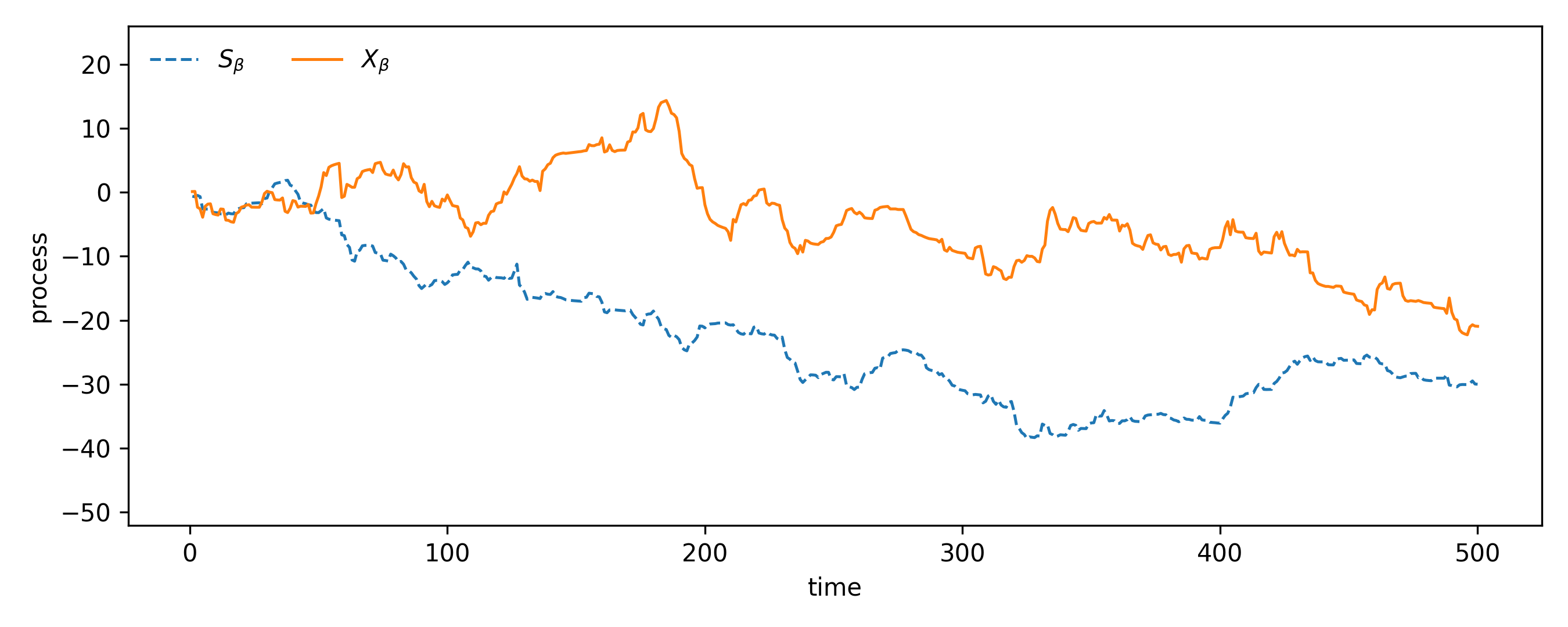}}}
 \end{figure}

 \begin{figure}[H]   \caption{Logarithmic vs polynomial shot-noises, for $\beta=0.3$}
{{\includegraphics[width=10.5cm]
    {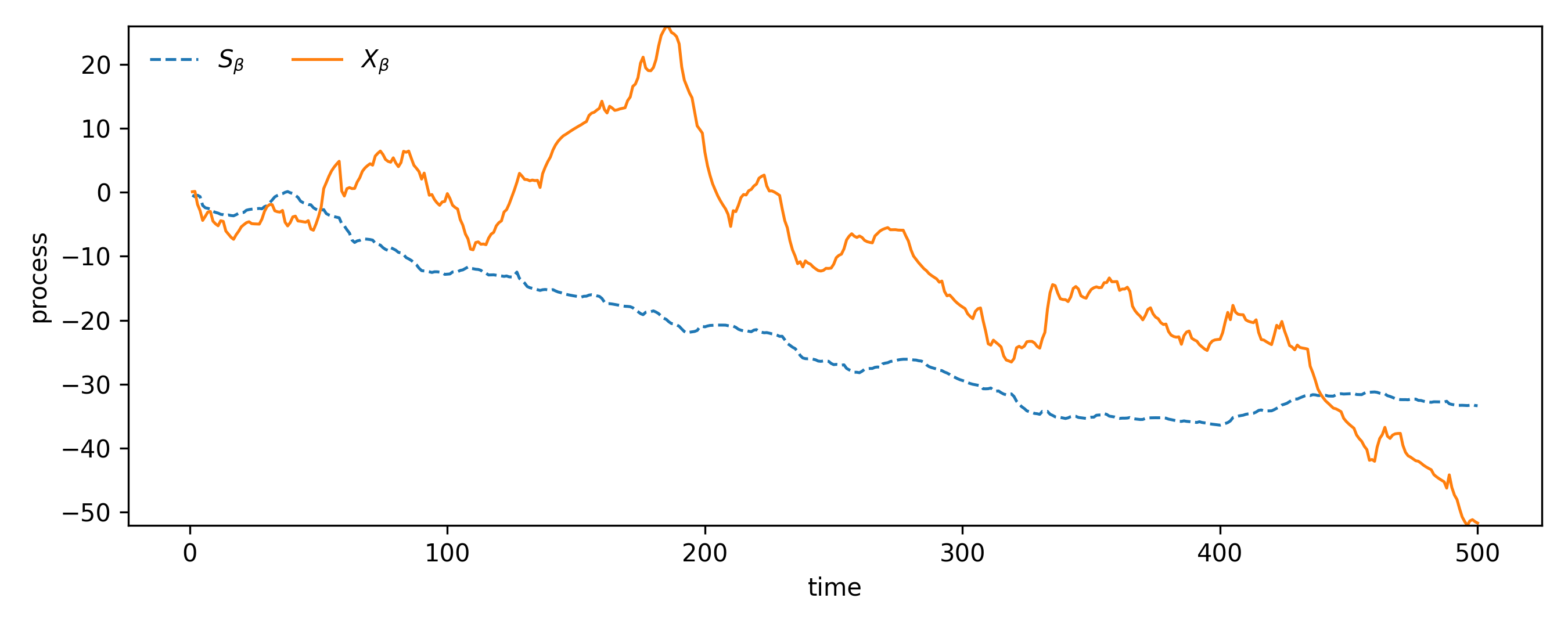}}}
    \label{fig:example}%
    \end{figure}

\end{remark}

\subsection{Asymptotic results}
We now study the limiting behavior of the shot-noise process defined in \eqref{defs2}, under a proper scaling. In order to obtain the result given in \cite{BEG} in the limit, we will hereafter set $\beta=(\alpha-1)/2$, for $\alpha \in (1,2)$. Note that, to simplify the notation, we adopt the following convention hereafter: $\log_+(\cdot):=(\log(\cdot))_+$.

\begin{theorem}\label{THM1}
Let $K_{\alpha,\lambda}:=K \lambda \Gamma(\alpha)$. If $K_2(\cdot)$ and $K_4(\cdot)$ are positive, uniformly bounded
functions such that $\lim_{t \to \infty}K_2(t)=K>0$ and $K_4(u)/K_2(u) \leq \kappa$, for any $u \in \mathbb{R} $ and $\kappa >0$, then the scaled shot-noise process defined as $\hat{S}_{\alpha,c}:=\left\{ \hat{S}_{\alpha, c}(t)\right\}_{t \geq 0}$, where
\begin{equation}
    \hat{S}_{\alpha,c}(t):=\frac{S(ct)}{\sqrt{cK_{\alpha,\lambda}}}=\frac{1}{\sqrt{cK_{\alpha,\lambda}}}\sum_{j=1}^{+\infty} \left(\log(ct)-\log(T_j)\right)^\frac{\alpha-1}{2}_+R_j(T_j), \qquad t \geq 0, \quad \alpha \in (1,2), \label{sn}
\end{equation}
converges, in the sense of finite-dimensional distributions, to the H-fBm $B^H_\alpha$, as $c \to +\infty$.
\end{theorem}
\begin{proof}
    By considering \eqref{sn}-\eqref{car}  and recalling that, by assumption, $K_1(u)=0$, for any $u$, we can write the $n$-point characteristic function of $\hat{S}_{\alpha,c}$, for
    $0<t_1<...<t_n<t$, $n \in \mathbb{N}$, and $k_j \in \mathbb{R},$ $j=1,...,n$, as follows
\begin{eqnarray}
    &&\mathbb{E}e^{i\sum_{j=1}^{n}k_j \hat{S}_{\alpha,c}(t_j)}=\mathbb{E}\exp \left\{\frac{i}{\sqrt{cK_{\alpha,\lambda}}}\int_0^{+\infty} \int_{\mathbb{R}} \sum_{j=1}^{n}k_j \log_+^{(\alpha -1)/2}(ct_j/u) r N(du,dr) \right\} \label{car2} \\
     &=& \exp \left\{ \lambda\int_0^{+\infty} \int_{\mathbb{R}} \left[e^{  \frac{i}{\sqrt{cK_{\alpha,\lambda}}} \sum_{j=1}^{n}k_j \log_+^{(\alpha -1)/2}(ct_j/u) r}-1-\frac{i}{\sqrt{cK_{\alpha,\lambda}}} \sum_{j=1}^{n}k_j \log_+^{(\alpha -1)/2}(ct_j/u) r \right]F_u(dr)du 
    \right\} \notag \\
    &=& \exp \left\{ \lambda c\int_0^{+\infty} \int_{\mathbb{R}} \left[e^{ \frac{i}{\sqrt{cK_{\alpha,\lambda}}} \sum_{j=1}^{n}k_j \log_+^{(\alpha -1)/2}(t_j/u) r}-1- \frac{i}{\sqrt{cK_{\alpha,\lambda}}} \sum_{j=1}^{n}k_j \log_+^{(\alpha -1)/2}(t_j/u) r \right]F_{cu}(dr)du 
    \right\}. \notag
\end{eqnarray}

Therefore, since it is known that $e^{ix}-1 -ix \sim -x^2/2$, as $x \to 0$, we obtain that, for $c \to \infty$,
\begin{eqnarray} \notag
\log \mathbb{E}e^{i\sum_{j=1}^{n}k_j \hat{S}_{\alpha,c}(t_j)} &\sim& -\frac{1 }{2K\Gamma(\alpha)} \int_0^{+\infty} \int_{\mathbb{R}} \left[\sum_{j=1}^{n}k_j \log_+^{(\alpha -1)/2}(t_j/u)\right]^2 r^2 F_{cu}(dr)du \\
&=& -\frac{1}{2K\Gamma(\alpha)} \int_0^{+\infty}  \left[\sum_{j=1}^{n}k_j \log_+^{(\alpha -1)/2}(t_j/u)\right]^2 K_2(cu)du. \label{arr1}
\end{eqnarray}
 By recalling \eqref{cf4} and \eqref{sigma}, in order to prove that 
 \[\lim_{c \to \infty}\log \mathbb{E}e^{i\sum_{j=1}^{n}k_j \hat{S}_{\alpha,c}(t_j)}=\log \mathbb{E}e^{i\sum_{j=1}^{n}k_j B^H_{\alpha }(t_j)}=-\frac{1}{2}\sum_{j,l=1}^{n}k_j k_l \sigma^\alpha_{j,l}   ,\]
 for any $t_j,k_j$, $j=1,...,n$ and $n \in \mathbb{N}$, it is sufficient to rewrite the integral in  \eqref{arr1} as follows
 \begin{eqnarray} \label{arr2}
&&\int_0^{+\infty}  \left[\sum_{j=1}^{n}k_j \log_+^{(\alpha -1)/2}(t_j/u)\right]^2 K_2(cu)du \\
&=& \sum_{j=1}^{n}k^2_j  \int_0^{t_j} \log^{\alpha -1}(t_j/u) K_2(cu)du+\sum_{\substack{j,l \in \left\{1,...,n \right\} \\
j \neq l}}k_j k_l \int_0^{t_j \wedge t_l} \log^{(\alpha -1)/2}(t_j/u) \log^{(\alpha -1)/2}(t_l/u) K_2(cu)du. \notag
\end{eqnarray}
By the assumptions on $K_2(\cdot)$ together with the dominated convergence theorem, we have that
 \begin{eqnarray} \label{arr3}
&& \lim_{c \to \infty}\log \mathbb{E}e^{i\sum_{j=1}^{n}k_j \hat{S}_{\alpha,c}(t_j)}\\
&=&- \frac{1}{2\Gamma(\alpha)} \sum_{j=1}^{n}k^2_j  \int_0^{t_j} \log^{\alpha -1}(t_j/u) du- \frac{1}{2\Gamma(\alpha)} \sum_{\substack{j,l \in \left\{1,...,n \right\} \\
j \neq l}}k_j k_l \int_0^{t_j \wedge t_l} \log^{(\alpha -1)/2}(t_j/u) \log^{(\alpha -1)/2}(t_l/u)du \notag \\
&=&- \frac{1}{2} \sum_{j=1}^{n}k^2_j t_j - \frac{C_\alpha}{2}  \sum_{\substack{j,l \in \left\{1,...,n \right\} \\
j \neq l}}k_j k_l (t_j \wedge t_l)\Psi\left(\frac{1-\alpha}{2}, 1-\alpha ; \log\left(\frac{t_j \vee t_l}{ t_j \wedge t_l}\right) \right)  \notag \\
&=& -\frac{1}{2}\sum_{j,l=1}^{n}k_j k_l \sigma^\alpha_{j,l} \notag
\end{eqnarray}
(details on the calculations are available in \cite{BEG}, Theorem 3.1 and formula (3.12)).
\end{proof}

We now prove the weak convergence of $\hat{S}_{\alpha,c}:=\left\{ \hat{S}_{\alpha, c}(t)\right\}_{t \geq 0}$ in the space $(\mathbb{D}[0,T],J_1)$, i.e. the space of all the $\mathbb{R}$-valued c\`{a}dl\`{a}g functions defined on the interval $[0,T]$, equipped with the Skorokhod $J_1$ topology with the metric $d_{J_1}$. To this aim we use the following result proved in \cite{PAN2}, Theorem 5.3, which provides sufficient conditions for such convergence.

\begin{lem}[\cite{PAN2}]
    Let $\left\{ X_n\right\}_{n \geq 1}$ and $X$ be stochastic processes with sample paths in $\mathbb{D}$ and let $\mu(\cdot)$ be a non-negative, set function from the Borel subset of $\mathbb{R}^+$ into $\mathbb{R}^+ \cup \left\{+ \infty \right\}$, such that $\mu(\varnothing)=0$. Assume that $\mu(\cdot)$ is monotone (i.e. $\mu(A) \leq \mu(B)$, for $A \subseteq B \subset \mathbb{R}^+$) and super-additive (i.e. $\mu(A)+\mu(B) \leq \mu( A \cup B)$, for any disjoint Borel sets $A$ and $B$). Then, under the following assumptions 

    \begin{enumerate}[label=(\emph{\roman*})]
        \item the finite dimensional distributions of the sequence $\left\{ X_n\right\}_{n \geq 1}$ converge to those of $X$;
        \item for any $\epsilon>0$, $\lim_{\delta \to 0} \mathbb{P}(|X(T)-X(T-\delta)| \geq \epsilon)=0$;
        \item there exists a constant $C>0$ such that, for any $0 \leq r \leq s \leq t \leq T$, with $t-r<2 \delta$, for some $\delta>0$ and for any $n \geq 1$, the following inequality holds
        \begin{equation}\label{eq}
        \mathbb{P}(|X_n(r)-X_n(s)| \wedge |X_n(s)-X_n(t)| \geq \epsilon) \leq \frac{C}{\epsilon^{4\beta}}\mu^{2 \eta}(r,t],
        \end{equation}
        for $\eta >1/2,$ $\beta \geq 0$;
        \item $\mu(0,t]$ is continuous in $t$;
    \end{enumerate}
    we have that $X_n \Rightarrow X$ in $(\mathbb{D}[0,T],J_1)$, for any $T>0$, as $n \to \infty.$
\end{lem}

\begin{theorem}
    Under the assumptions given in Theorem \ref{THM1}, the scaled shot-noise  defined in \eqref{sn} weakly converges to the H-fBm, i.e. 
    \[
    \hat{S}_{\alpha,c} \Rightarrow B_\alpha^H, \qquad \text{in } (\mathbb{D}[0,T],J_1).
    \]
\end{theorem}
\begin{proof}
Let us take $\mu(\cdot,\cdot)$ equal to the increment variance of $B_\alpha^H$, i.e. $\mu(r,t]:=\rho (r,t) $ given in \eqref{rho}. Then $\mu(\cdot,\cdot)$ is non-negative, monotone and super-additive by (\emph{i}), (\emph{iii}) and (\emph{ii}) of Lemma \ref{lem1}, respectively.

The process $ \hat{S}_{\alpha,c}$ has sample paths in $\mathbb{D}[0,T]$, since the sum $\sum_{j=1}^{N(t)} \left(\log(ct)-\log(T_j)\right)^\frac{\alpha-1}{2}R_j(T_j)$ has at most finitely many jumps in any compact time interval.  Moreover, it is already proved in \cite{BEG} that the same holds for the limiting process $B_\alpha^H$.

We now prove that the assumptions (\emph{i})-(\emph{iv}) hold for $ \hat{S}_{\alpha,c}$ and $B_\alpha^H$:
   \begin{enumerate}[label=(\emph{\roman*})]
\item the convergence of the finite dimensional distributions of $ \hat{S}_{\alpha,c}$ to those of $B_\alpha^H$ was proved in Theorem \ref{THM1};
\item this condition is verified due to stochastic continuity of $B_\alpha^H$ proved in \cite{BEG}, Theorem 3.2;
\item we now check that condition \eqref{eq} is satisfied for $\eta=\beta=1$ and $\mu(r,t]=\rho (r,t) $: we can write, by Markov and H\"{o}lder inequalities,
 \begin{eqnarray}\label{eq2}
        \mathbb{P}(|\hat{S}_{a,c}(r)-\hat{S}_{a,c}(s)|  \wedge |\hat{S}_{a,c}(s)-\hat{S}_{a,c}(t)| \geq     
        \epsilon)
         &\leq& \frac{1}{\epsilon^4} \mathbb{E}(|\hat{S}_{a,c}(r)-\hat{S}_{a,c}(s)| \wedge |\hat{S}_{a,c}(s)-\hat{S}_{a,c}(t)|)^4  \\
         &\leq&  \frac{1}{\epsilon^4} \mathbb{E}(|\hat{S}_{a,c}(r)-\hat{S}_{a,c}(s)|^2|\hat{S}_{a,c}(s)-\hat{S}_{a,c}(t)|^2) \notag \\
        &\leq& \frac{1}{\epsilon^4} (\mathbb{E}|\hat{S}_{a,c}(r)-\hat{S}_{a,c}(s)|^4)^{1/2}(\mathbb{E}|\hat{S}_{a,c}(s)-\hat{S}_{a,c}(t)|^4)^{1/2}. \notag 
        \end{eqnarray}
The fourth-order moment of the increment can be evaluated as follows, by denoting, for brevity, $b(c):=\sqrt{cK_{\alpha,\lambda} }$, for $0 \leq s \leq t$,
\begin{eqnarray}
    \mathbb{E}\left[ \hat{S}_{a,c} (t)-\hat{S}_{\alpha,c} (s)\right]^4 
&=&\frac{1}{b(c)^4}\mathbb{E}\left[ \int_0^{+\infty} \int_{\mathbb{R}} \left(\log_+^{(\alpha -1)/2}(ct/u)-\log_+^{(\alpha -1)/2}(cs/u)\right) r N(du,dr)\right]^4 \notag \\
&=&\frac{1}{b(c)^4}\mathbb{E}\left[ \int_0^{cs} \int_{\mathbb{R}} \left(\log^{(\alpha -1)/2}(ct/u)-\log^{(\alpha -1)/2}(cs/u) \right) r N(du,dr)\right. + \notag \\
&+& \left. \int_{cs}^{ct} \int_{\mathbb{R}} \log^{(\alpha -1)/2}(ct/u)r N(du,dr)\right]^4  \notag 
\end{eqnarray}
Then, by applying Lemma 3.3 in \cite{PAN}, with $(\log ct-\log u)^{(\alpha -1)/2}_+$ in place of $g(ct-u)$, we can write that, for $0 \leq s \leq t$,
\begin{eqnarray}
    &&\mathbb{E}\left[ \hat{S}_{a,c} (t)-\hat{S}_{\alpha,c} (s)\right]^4 \label{E4} \\
&=&\frac{3}{b(c)^2}\mathbb{E}\left[ \int_0^{t}  \left(\log_+^{(\alpha -1)/2}(t/u)-\log_+^{(\alpha -1)/2}(s/u)\right)^2 K_2(cu) \overline{N}_c(du)\right]^2 +\notag \\
&+&\frac{1}{b(c)^4}\mathbb{E}\left[ \int_0^{t}  \left(\log_+^{(\alpha -1)/2}(t/u)-\log_+^{(\alpha -1)/2}(s/u)\right)^4 \left(K_4(cu) -3K_2(cu)^2\right) \overline{N}_c(du)\right],  \notag 
\end{eqnarray}
where we have considered that $\log_+^{(\alpha -1)/2}(s/u)=0$, for $u>s$ and that $K_4(u):=\int_{\mathbb{R}}x^4 dF_u(x)$, which is finite, by assumption. 

We now obtain a uniform bound for the fourth moment in \eqref{E4}: let
\[
f_{s,t}(u) :=
\log_+^{(\alpha-1)/2}(t/u)
-
\log_+^{(\alpha-1)/2}(s/u),
\]
so that, by using standard moment formulas for Poisson integrals (see \cite{KIN}, p.27), we have that
\begin{align*}
\mathbb{E}\Bigg[\int_0^t f_{s,t}(u)^2 K_2(cu)\,N_c(du)\Bigg]^2
&\le C' \left[
c \int_0^t f_{s,t}(u)^4\,du
+
c^2 \left(\int_0^t f_{s,t}(u)^2\,du\right)^2
\right],
\end{align*}
where we used the assumption that $K_2(\cdot)$ is uniformly bounded. Recalling that $b(c)^2 = cK_{\alpha,\lambda}$, it follows that
\[
\frac{1}{b(c)^2}
\mathbb{E}\Bigg[\int_0^t f_{s,t}(u)^2 K_2(cu)\,N_c(du)\Bigg]^2
\le
C \left[
\frac{1}{c}\int_0^t f_{s,t}(u)^4\,du
+
\left(\int_0^t f_{s,t}(u)^2\,du\right)^2
\right].
\]

Similarly, using the assumption $K_4(u)/K_2(u)\le \kappa$, we obtain
\[
\frac{1}{b(c)^4}
\mathbb{E}\Bigg[
\int_0^t f_{s,t}(u)^4 \big(K_4(cu)-3K_2(cu)^2\big)\,N_c(du)
\Bigg]
\le
 \frac{C''}{c}\int_0^t f_{s,t}(u)^4\,du.
\]

Therefore, for $c$ sufficiently large,
\[
\mathbb{E}\big|\hat S_{\alpha,c}(t)-\hat S_{\alpha,c}(s)\big|^4
\le
C''' 
\left(\int_0^t f_{s,t}(u)^2\,du\right)^2
.
\]

Finally, since $\int_0^t f_{s,t}(u)^2\,du = \Gamma(\alpha)\rho(s,t)$ by \eqref{rho},
and the second term is negligible, we obtain
\[
\mathbb{E}\big|\hat S_{\alpha,c}(t)-\hat S_{\alpha,c}(s)\big|^4
\le \bar C\,\rho(s,t)^2,
\]
with $\bar C>0$ independent of $c$.

 Therefore, we can write that, for $\bar C>0$, 
\[
\mathbb{P}(|\hat{S}_{a,c}(r)-\hat{S}_{a,c}(s)|  \wedge |\hat{S}_{a,c}(s)-\hat{S}_{a,c}(t)| \geq \epsilon) \leq \frac{\bar C}{\epsilon^4} \rho(r,s) \rho(s,t)\leq \frac{\bar C}{\epsilon^4} \rho(r,t)^2,
\]
where the last step follows by the property \textit{(iii)} of $\rho(\cdot, \cdot)$, proved in Lemma \ref{lem1}.
\item $\rho(0,t)$ is continuous in $t$ because of point \textit{(v)} in Lemma \ref{lem1}.
   \end{enumerate}
\end{proof}

\section*{Acknowledgements}

Luisa Beghin and Enrico Scalas acknowledge financial support under NRRP, Mission 4, Component 2, Investment 1.1, Call for tender No. 104 published on 2.2.2022 by the Italian MUR, funded by the European Union – NextGenerationEU– Project Title {\em Non–Markovian Dynamics and Non-local Equations} – 202277N5H9 - CUP: D53D23005670006.

E. Scalas also acknowledges financial support under the project: \\
000317\_24\_RICERCA\_UNIV\_2023\_PROG\_MEDI\_SCALAS - RICERCA ATENEO 2023 - SCALAS
PROGETTI MEDI. The title of the project is {\em Approximation of stochastic processes by
means of sums of random telegraph processes}.

\end{document}